\documentclass[10pt,journal]{article}

\usepackage[utf8]{inputenc}

\usepackage{hyperref}
\usepackage{verbatim} 

\usepackage{graphicx}
\usepackage{caption}
\usepackage{subcaption}
\usepackage{amssymb}
\usepackage{ textcomp } 
\usepackage{color}
\usepackage{enumerate}
\usepackage{mathrsfs}

\usepackage[ruled,vlined, onelanguage]{algorithm2e}

\usepackage{amsopn}

\usepackage{amsmath}
\usepackage{amsthm}
\usepackage{amssymb}

\usepackage{multicol}

\newcommand{\C}{\mathbb{C}}

\newcommand{\R}{\mathbb{R}}

\newcommand{\calA}{\mathcal{A}}

\newcommand{\calC}{\mathcal{C}}

\newcommand{\calH}{\mathcal{H}}

\newcommand{\calM}{\mathcal{M}}
\newcommand{\calP}{\mathcal{P}}

\newcommand{\calS}{\mathcal{S}}
\newcommand{\calX}{\mathcal{X}}
\newcommand{\calY}{\mathcal{Y}}
\newcommand{\calZ}{\mathcal{Z}}

\newcommand{\scrS}{\mathscr{S}}

\newcommand{\abs}[1]{\vert #1 \vert}
\newcommand{\norm}[1]{\Vert #1 \Vert}
\newcommand{\abss}[1]{\left\vert #1 \right\vert}
\newcommand{\normm}[1]{\left\Vert #1 \right\Vert}
\newcommand{\set}[1]{\left\lbrace #1\right\rbrace}
\newcommand{\sse}{\subseteq}
\newcommand{\sprod}[1]{\left\langle #1 \right\rangle}

\newcommand{\sph}{\mathbb{S}}

\newcommand{\prob}[1]{\mathbb{P}\left( \text{#1}\right)}
\newcommand{\prb}[1]{\mathbb{P}\left( #1 \right)}
\newcommand{\erw}[1]{\mathbb{E}\left( #1 \right)}

\usepackage{dsfont}

\newcommand{\geqsim}{\gtrsim}
\newcommand{\leqsim}{\lesssim}

\DeclareMathOperator{\supp}{supp}
\DeclareMathOperator{\id}{id}

\newcommand{\muminus}{\mu_{-}}
\newcommand{\muplus}{\mu_{+}}

\newtheorem{lem}{Lemma}

\newtheorem{theo}[lem]{Theorem}
\newtheorem{cor}[lem]{Corollary}

\newtheorem{rem}[lem]{Remark}

\numberwithin{lem}{section}

\title{Sparse Blind Deconvolution and Demixing Through $\ell_{1,2}$-Minimization.\footnote{This article has been accepted for publication in \emph{Adv. Comput. Math.} The final publication is available at Springer via \href{http://dx.doi.org/10.1007/s10444-017-9533-0}{{\tt DOI 10.1007/s10444-017-9533-0 }}}}

\usepackage[affil-it]{authblk}

\author{Axel Flinth\thanks{Postal adress: Techische Universität Berlin, Sekretariat MA 5-4, Straße des 17 Juni 136, 10623 Berlin. \\ E-mail: {\tt flinth@math.tu-berlin.de}}}
\affil{Institut für Mathematik \\ Technische Universität Berlin}

\begin{document}

\maketitle

\begin{abstract}
This paper concerns solving the sparse deconvolution and demixing problem using $\ell_{1,2}$-minimization. We show that under a certain structured random model, robust and stable recovery is possible. The results extend results of Ling and Strohmer [\emph{Self Calibration and Biconvex Compressive Sensing}, Inverse Problems, 2015], and in particular theoretically explain certain experimental findings from that paper. Our results do not only apply to the deconvolution and demixing problem, but to recovery of column-sparse matrices in general. 

\end{abstract}

  \subsection*{Acknowledgement}
	The author acknowledges support from  Deutsche
Forschungsgemeinschaft (DFG) Grant KU 1446/18~-~1 and the Berlin Mathematical School. He thanks Felix Krahmer and Dominik St\"oger for pointing out weak spots in the first version of the article, as well as providing references making it possible to repair them -- and even enhance the results slightly in the process.
He also wishes to thank Gitta Kutyniok for fruitful discussions.

\section{Introduction}
Assume that we observe a vector $v \in \C^q$\footnote{In a previous (preprint) version of this article, all (random) vectors appearing were assumed to be real. As was pointed out to the author, this is not very fortunate due to the use of the Fourier transform. Concretely, the meaning of the vectors $a_i$ being Gaussian distributed changes meaning into something more artificial than when they are kept complex. For this reason, a change to a complex setting was made.} and are told that it is a sum of $r$ convolutions of $r$ pairs of vectors $w_i, z_i$, i.e.
\begin{align*}
	v= \sum_{i \in [r]} w_i * z_i,
\end{align*} 
where $[r]$ is a short-hand notation for the set $\set{1, 2, \dots, r}$. This problem is known as the \emph{blind deconvolution and demixing problem} (we need to 'demix' each contribution $w_i * z_i$ from the sum $\sum_{i} w_i * z_i$, as well as  'deconvolve' the unknown filters $w_i$ to recover the vectors $z_i$). In general, it is of course impossible to reconstruct the pairs $(w_i, z_i)$ without any structural assumptions on them. In this work, we will assume that there exist (known) subspaces $W_i$ and $U_i$ of $\C^q$, $i \in [r]$ such that $w_i \in W_i$ and $z_i \in U_i$ for each $i \in [r]$. This could in a communication application correspond to filters $w_i$ and signals $z_i$ having certain bandwidth restrictions.

There is a standard way to transform the blind deconvolution problem  into a matrix recovery problem (Ahmed, Recht and Romberg 2014 \cite{ahmed2014blind}; Ling and Strohmer 2015 \cite{ling2015sparseDeconv}). For certain sparsity assumptions on the vectors $w_i$ and $z_i$, this results in a recovery problem of a \emph{column sparse matrix tuple} $\calZ$. Such problems will be the focus of this paper.  

Before discussing strategies for solving such problems, let us begin by describing the transformation procedure in detail. Taking the Fourier transform of the above equation, we arrive at
\begin{align*}
	\hat{v}= \sum_{i \in [r]} \hat{w}_i \odot \hat{z}_i,
\end{align*}
where $\odot$ denotes elementwise multiplication, i.e. $(x \odot y)_i = x_i y_i$. Let us introduces the bases $(B^i_\ell)_{\ell \in [k_i]}$ for  $\widehat{W}_i$, and $(A_\ell^i)_{\ell \in [n_i]}$ for $\widehat{U}_i$, $i \in [r]$ (note that we do not assume that all subspaces have the same dimension). Then there exists coefficients $(f^i_k)_{k\in [k_i]}$ and $(g^i_j)_{j \in [n_i]}$, $i \in [r]$, such that for all $\ell \in [q]$
\begin{align*}
 \hat{v}_\ell &= \sum_{i \in [r]} \left(\sum_{\kappa \in [k_i]} f^i_\kappa B^i_\kappa\right)_\ell \left(\sum_{j \in [n_i]} g^i_j A^i_j \right)_\ell \\
	&= \sum_{i \in [r]} \sum_{\kappa \in [k_i], j \in [n_i]}  B^i_\kappa (\ell) f^i_\kappa g^i_j A^i_j(\ell) 
\end{align*}
If we define $r$ matrices $Z_i \in \C^{k_i,n_i}$ through  $Z_i = f^i (\overline{g}^i)^*$, i.e. $Z_i(\kappa,j)=f^i_\kappa g^i_j$, we see that we can write the latter sum as
\begin{align*}
	\sum_{i \in [r]} \sum_{\kappa \in [k_i]}  B^i_\kappa (\ell) (Z_i A^i_{(\cdot)}(\ell))_\kappa = \sum_{i \in [r]} \sprod{b^i_\ell, Z_i a^i_\ell},
\end{align*}
where we defined new vectors $a_\ell^i \in \C^{n_i}$ and $b_\ell^i \in \C^{k_i}$ through $$a_\ell^i(j) := A^i_j(\ell), \quad b_\ell^i(\kappa) := \overline{B^i_\kappa(\ell)}.$$

Hence, we have rewritten our deconvolution problem to a problem of recovering a matrix tuple $\calZ= (Z_i) \in \bigoplus_{i \in [r]} \C^{k_i, n_i}$ from the measurements
\begin{align*}
	\hat{v} = \left( \sum_{i \in [r]} \sprod{b^i_\ell, Z_i a^i_\ell} \right)_{\ell \in [q]}.
\end{align*}
Let us denote the linear \emph{measurement map} $$\bigoplus_{i \in [r]} \C^{k_i,n_i} \to \C^q, \quad \calZ \to \left( \sum_{i \in [r]} \sprod{b^i_\ell, Z_i a^i_\ell} \right)_{\ell \in [q]}$$ by $\calA$. \newline

In this paper, we want to assume that the basis coefficients $g^i$ are sparse, a situation considered also in (Ling and Strohmer 2015 \cite{ling2015sparseDeconv}). We do not, however, pose any requirements on $f^i$. The sparsity of the $g^i$-coefficients has the consequence that the matrices $Z_i = f^i(g^i)^*$ are \emph{column-sparse}, since only the columns corresponding to indices in $\supp g^i$ are not equal to zero.  Since it is well-known that this structure is promoted by the $\ell_{1,2}$ - norm  (Eldar and Mishali 2009 \cite{Eldar2009BlockSparse}; Stojnic, Parvaresh and Hassibi 2009 \cite{stojnic2008reconstruction})
\begin{align*}
		\norm{M}_{1,2} = \sum_{i=1}^n \norm{M(i)}_2,
		\end{align*} 
		where $M(i)$ is the $i$:th column of $M$, this naturally calls for the following recovering procedure
	\begin{align}
		\min \norm{\calZ}_{1,2}:= \sum_{i \in [r]} \norm{Z_i}_{1,2} \text{ subject to } \calA(\calZ)=b. \tag{$\calP_{1,2}$}
	\end{align}
	Despite this approach arguably being canonical for recovering column-sparse matrices, there has not been any theoretical analysis of the program $\calP_{1,2}$ when the measurement map $\calA$ is as above. 
	
	The article (Ling and Strohmer 2015 \cite{ling2015sparseDeconv}) has provided (in the case that $r=1$) a discussion on \emph{$\ell_1$-minimization} for recovering $\calZ$
	\begin{align}
		\min \norm{Z}_{1} \text{ subject to } \calA(Z)=b, \tag{$\calP_{1}$}
	\end{align} 
	where the $\ell_1$-norm of a matrix is simply defined as the sum over the absolute values its entries. In particular, they recover the well-known asymptotic result that $m \geqsim sk\log(nk)$ measurements suffices for $\calP_1$ to be successful at recovering an $sk$-sparse matrix in $\C^{k,n}$. At the end of the paper, they perform numerical experiments which show that $\ell_{1,2}$-minimization actually performs better than $\ell_1$-minimization at recovering column-sparse matrices. They claim that they have theoretical guarantees also for the $\calP_{1,2}$-problem, but have as of today not yet published any.  \newline
	
	In this work, we will, to some extent, provide that missing theoretical explaination, by generalizing the results of (Ling and Strohmer 2015, \cite{ling2015sparseDeconv}). We even improve them a bit by additionally including an argument for stability of the problem, in the sense that approximately column-sparse matrices will be approximately recovered by $\calP_{1,2}$. Also, our results are a bit more robust to noise. In addition to this, we consider the general case of $r>1$, i.e., we include the demixing part of the problem. 
	
	Although the main route of many of the arguments are the same as in the mentioned paper, several adjustments has had to be made for the argument to work also for the $\ell_{1,2}$-case. Many proofs become more difficult from a technical point of view. Also, a dual certificate type of condition for stability and robustness of $\ell_1$-minimization from (Foucart and Rauhut  2013 \cite{MathIntroToCS}) has had to be generalized to $\ell_{1,2}$-minimization.
	
 The remainder of this introduction is devoted to discussing related work. The rest of the paper will then be organized as follows: in Section \ref{sec:MainResult}, we present the measurement model we use, our main result, as well as an outline of the proof of it.  The details of the proofs are postponed to Section \ref{sec:Proofs}.	
	
	\subsection{Related work}
	$\ell_{1,2}$-minimization has been theoretically analysed before, i.e. in the resources 	(Eldar and Mishali 2009 \cite{Eldar2009BlockSparse}; Stojnic, Parvaresh and Hassibi 2009 \cite{stojnic2008reconstruction}) mentioned above. In (Asi, Mantzel and Romberg 2009 \cite{Romberg2009channelProtect}), it is mentioned as a way of solving the blind deconvolution problem, explicitly with applications in communication in mind.

	$\ell_{1,2}$-minimization is by no means the only possible way to recover the matrix tuple $\calZ$. In fact, one can alternatively use the fact each matrix $Z_i = f^i (g^i)^*$ is \emph{low-rank} as a prior. This naturally calls for nuclear norm-minimization (Fazel, Recht and Parrilo 2010 \cite{recht2010guaranteed}). The nuclear norm of a matrix tuple is canonically defined as the sum of the nuclear norms $\norm{Z_i}_{*}$ of each matrix $Z_i$. The nuclear norm of a matrix $Z_i$ is thereby given by the sum of the singular values of $Z_i$. Hence,
	\begin{align*}
		 \norm{\calZ}_{*}=\sum_{i \in [r]} \norm{Z_i}_{*} = \sum_{i \in [r]} \sum_{j\in [n_i]} \sigma_j(Z_i). 
	\end{align*}
The paper (Ling and Strohmer 2015 \cite{strohmer2015demixing}) treats recovery of low-rank matrix tuples using nuclear norm-minimization: they are able to prove that under several technical conditions resembling the ones we will consider in this paper, $\frac{q}{\log(q)^3}  \geqsim r^2 \max(k,n) \log(r+1)$ is enough to secure that the nuclear norm minimization procedure recovers the correct matrix tuple. In (Jung, Krahmer and Stöger 2016 \cite{Stoeger:cosera16}), it was even shown, by different authors, that the quadratic dependence on $r$ can be removed (to be precise, their result requires $\frac{q}{log(q)^2} \geqsim r(k\log^2(k) + n)\log\left(r\sqrt{n\log(nq) + \log(q)}\right)$).

 However, as was pointed out in (Ling and Strohmer 2015 \cite{ling2015sparseDeconv}), numerical experiments show that both $\ell_{1}$- and $\ell_{1,2}$-minimization perform significantly better than nuclear norm minimization when it comes to recovering matrices with the structure described above.  This is not hard to argue heuristically: Assuming $r=1$, the above approach tries to recover a $\C^{k,n}$-matrix of rank $1$, which needs an order of $n+k$ measurements. On the other hand $\calP_{1,2}$ (or $\calP_1$) tries to recover an $s$-column sparse (or $sk$-sparse) $\C^{k,n}$ matrix , which only needs $sk\log\left( nk\right)$. For really small sparsities and moderate $k$, $sk$ can be smaller than $n+k$. Therefore, the authors of the mentioned paper concentrate their efforts, as will we, on analysing the $\ell_1$-minimization (and $\ell_{1,2}$-minimization, respectively). \newline

 As observant readers already may have pointed out, the ''true dimension'' of the problem of recovering a matrix $f g^*$ with $f \in \R^k$ and $g \in \R^n$ $s$-sparse is neither $s\cdot k$ nor $n+k$, but instead $s+k$ (a mathematically precise statement of this claim is provided in (Kech and Krahmer 2016 \cite{Kech2016Bilinear})). As of today, to the best knowledge of the author, there are no convex minimization procedures which succeed which such few measurements. In this context, ( Bresler, Lee and Wu 2013 \cite{lee2013near}) should be mentioned. In that paper, the authors describe a alternating minimization procedure which under some additional conditions on the vectors $f$ and $g$ succeed with high probability already when $q$ is of the order of $s+k$.


\section{Main Result} \label{sec:MainResult}

In this section, the main result together with an outline of the proof will be presented. In order to do that, we first need to describe our measurement model as well as the assumptions we make. 

Let us begin by describing the properties of the basis vectors $b^i_\ell$ and $a_\ell^i$.  The $b^i_\ell$ are assumed to be fixed and known (this corresponds to the spaces $W_i$ to be fixed and known). Also, we assume that each tuple $(b^i_\ell)_{\ell \in [q]}$ forms a Parseval Frame of $\C^{k_i}$, i.e
\begin{align*}
	\sum_{\ell \in [q]} b_\ell^i (b_\ell^i)^* = \id.
\end{align*}
In order for our proof to work, we will have to assume that the frames are somewhat well-conditioned. Concretely, we will assume that there exists positive numbers $\muminus$ and $\muplus$ so that $\tfrac{q}{k_i}\norm{b_\ell^i}_2^2 \in [\muminus^2, \muplus^2]$ for all $i$ and $\ell$.  

The vectors $a_\ell^i$ are assumed to be known, but not fixed. Rather, we assume them to be independent Gaussian vectors in their respective spaces. The spaces $U_i$ from above are hence uniformly randomly chosen, and the bases of them as well. A statement of the form ''the solution of $\calP_{1,2}$ is equal to the ground truth signal with high probability'' will hence mean that the method works for a very high fraction of possible subspaces $U_i$ and sparsifying transforms in those respective spaces.


The signal model is as follows: we consider matrix tuples $\calZ = (Z_i)_{i \in [r]} \in \bigoplus_{i \in [r]} \C^{k_i,n_i}$, where each matrix $Z_i$ is assumed to be column-supported on some set $S_i$, i.e. only the columns $Z_i(j)$ for $j \in S_i$ are nonzero. $S_i$ has cardinality $s_i$. Alternatively, we will sometimes speak of matrix tuples being supported on sets $\calS= \bigotimes_{i\in [r]} S_i$, with the exact same meaning. Note that the $s_i$ are not assumed to be equal, and in particular, some of them can even be equal to zero (which corresponds to $Z_i =0$). To simplify the notation somewhat, we will use the following short-hands:
\begin{align*}
	s = &\sum_{i \in [r]} s_i, \quad n = \sum_{i \in [r]} n_i , \\
	 k^*=& \max_{i \in [r]} k_i , \quad k_* = \min_{i \in [r]} k_i
\end{align*}
 $\calP_{\calM}$ denotes the orthogonal projection on to the space of matrix tuples supported on $\calM \sse \bigotimes_{i \in [r]}[n_i]$. When convenient, we will also use the notation $\calX_\calS:= \calP_{\calS} \calX$.

It will at several places in the article come in handy to decompose the map $\calA$. We define $A^i : \C^{k_i, n_i} \to \C^q$ through
\begin{align*}
	A^i(Z) = \calA( (0, \dots, 0, \stackrel{i}{Z}, 0, \dots, 0) ),
\end{align*}
and also $A^i_j : \C^{k_i} \to \C^q$ through
\begin{align*}
	A^i_j\nu = A^i(\nu e_j^*).
\end{align*}

In our main result, we will assume the following asymptotics of the number of measurements $q$: 
\begin{align*}
	q \geqsim  \muplus^2 k^*s \log\left( 1+ \tfrac{\muplus^4 k^* \sum_{i \in [r]} s_ik_i}{\muminus^2k_*}\right)\cdot\log\left(\tfrac{nk^*}{\epsilon}\right).
\end{align*}
$\epsilon>0$ is an upper bound on the failure probability. This amount of measurements is, up to logarithmic term, more or less what could be expected: we are trying to recover a signal from a $\sum_{i \in [r]} s_i k_i$-dimensional structure embedded in a $\sum_{i \in [r]} n_i k_i$-dimensional space. Intuitively (compare for instance (Candes and Tao 2005 \cite{CandesTao2005})), this calls for 
\begin{align*}
	q \geqsim \big( \sum_{i \in [r]} s_i k_i\big) \log\big(\sum_{i \in [r]}n_i k_i\big)
\end{align*}
compressive measurements. Note that if the $k_i$ are not varying to much with $i$, we have $\sum_{i \in [r]} s_i k_i \approx k^* s$ and $\sum_{i \in [r]} n_i k_i \approx k^* n$. In particular, we have equality for $r=1$. \newline

To simplify reading the paper, let us summarize all of our assumptions in a list.

\begin{enumerate}[(a)]
	\item {\bf $a$-statistics.} The vectors $a^i_\ell \in \C^{n_i}$ are independent Gaussians.
	\item {\bf Parseval Frames.} For each $i$ we have
	\begin{align*}
	\sum_{\ell \in [q]} b_\ell^i (b_\ell^i)^* = \id.
\end{align*}

	\item {\bf $b$-norms under control:} There exists $\mu_{-}, \mu_{+} >0$ with
	\begin{align*}
		\tfrac{q}{k_i}\norm{b^i_\ell}_2^2 \in [\mu_{-}^2, \muplus^2]
\end{align*}
for each $i$ and $\ell$.
	 
	\item{\bf $q$-asymptotics } We have 
\begin{align*}
	q \geqsim  \muplus^2 k^*s \log\left( 1+ \tfrac{\muplus^4 k^* \sum_{i \in [r]} s_ik_i}{\muminus^2k_*}\right)\cdot\log\left(\tfrac{nk^*}{\epsilon}\right).
\end{align*}
\end{enumerate}

The main result of this paper reads as follows.
\begin{theo} \label{th:mainResult}
	Under the assumptions $(a)$ to $(d)$, every matrix tuple $\calZ_0$ with $Z_i^0$ $s_i$-column sparse is the unique solution of $\calP_{1,2}$ with a probability larger than $1-\epsilon$.
	
	In fact, $(a)$ to $(d)$ will even imply stable and robust recovery in the sense that for any matrix tuple $\calZ_0$ and $y = \calA(Z_0) + \nu$ with $\norm{\nu}_2 \leq \sigma$, with a probability larger than $1-3\epsilon$, any solution $\calZ^*$ of the program
	\begin{align*}
		\min \norm{\calZ}_{1,2} \text{ subject to } \norm{\calA(\calZ)-y}_2\leq \sigma
	\end{align*}
	obeys 
	\begin{align*}
		\norm{Z^*-\calZ_0}_F \leq C_1 \norm{\calP_{\calS^c}\calZ_0}_{1,2} + (C_2 + C_3 \sqrt{s})\sigma, 
	\end{align*}
	where $C_1$, $C_2$ and $C_2$ are universal constants.
\end{theo}

\begin{rem} \label{rem:12Better}
	The assumptions we have made are the same as in (Ling and Strohmer 2015 \cite{ling2015sparseDeconv}), except for the $(d)$-assumption. The mentioned paper only deals with the case $r=1$, but in that case, their equivalent of the $(d)$-assumption reads
	\begin{align}
		\frac{q}{\log^2(q)} \geqsim \muplus^2 ks \log(nk). \tag*{$(d)_*$}
	\end{align}
	If we put $\epsilon = q^{1-\alpha}$ in our assumption (d) (as is made in the mentioned article) and assume that $\tfrac{\muplus^2}{\muminus^2}$ is close to one (note that $k^*=k_*$ in the case $r=1$), we arrive at
	\begin{align*}
		q \geqsim \muplus^2 ks \log(1+ \muplus^2sk) \log(nkq^{\alpha-1}). \tag*{$(d)_{r=1}$}
	\end{align*} 
	Compared to $(d_*)$, we gain a $\log(q)$-term but lose a $\log(sk)$-term. Since $\log(q) \geqsim \log(1+\mu^2sk)$ for $q \geqsim sk$, we see that $(d)_*$ in fact implies $(d_{r=1})$. Hence, the $(d)_{r=1}$-assumption slightly weaker than $(d)_*$. The reason for this improvement is of proof-technical nature: We apply a strong version of the Matrix Bernstein inequality, and it is probably possible to obtain this rate also for $\ell_1$-minimization.
		
	Looking a bit closer, we however find a way in which $\ell_{1,2}$-minimization inevitably outperforms $\ell_1$-minimization: Under assumptions $(a)$ to $(c)$, together with $(d)_*$, the authors of the mentioned article prove that the regularized  $\ell_1$-minimization program, with a probability larger than $q^{1-\alpha}$ (the implicit constant is dependent on the parameter $\alpha$), obeys
	\begin{align*}
		\norm{Z^*-Z_0}_{F} \leq (C_1 + C_2 \sqrt{ks})\sigma
	\end{align*}
	for every $s$-column sparse $Z_0$. This error bound is worse than the one we prove for exactly column-sparse signals (since $\sqrt{ks}\geq\sqrt{s}$), and also does not account for small deviations from the sparsity assumption. Hence, our analysis indicates that $\ell_{1,2}$-minimization really works better than $\ell_{1}$-minimization for recovering $s$-column sparse matrices using the considered type of measurements.
\end{rem}

\subsection{Outline of the proof.}

The proof of Theorem \ref{th:mainResult} will inevitably be technically quite involved. In the following, we will describe its basic route. All details are given in the next section.  We again point out that we closely follow the paper (Ling and Strohmer 2015 \cite{ling2015sparseDeconv}).

The start of the argument is the following lemma. It is a generalization of (Foucart and Rauhut 2013 \cite[Th. 4.33]{MathIntroToCS}), which is a corresponding statement about $\ell_1$-minimization. It will make use of the regularized program 
\begin{align}
	\min \norm{ \calZ}_{1,2} \text{ subject to } \norm{\calA(\calZ)-y}_2 \leq \sigma	\tag{$\calP_{1,2}^\sigma$}.
\end{align}
To simplify the notation, let us introduce the short hand $\widehat{\calX}$ for the matrix tuple formed by normalizing each column of each submatrix of $\calX$. To be precise,
\begin{align*}
	\widehat{X}_i(j) := \begin{cases} \tfrac{X_i(j)}{\norm{X_i(j)}_2} &\text{ if } X_i(j) \neq 0 \\
	0 &\text{ else.} \end{cases}
\end{align*}

\begin{lem} Let $\calZ_0 \in \bigoplus_{i \in [r]}\C^{k_i,n_i}$ and let $\calS \sse \bigotimes_{i \in [r]}[n_i]$ be arbitrary. Consider a linear map $\calA$ from $\bigoplus_{i \in [r]} \C^{k_i,n_i} $ to $ \C^q$ and noisy measurements $y=\calA(Z_0)+n$ with $\norm{n}_2 \leq \sigma$.  \label{lem:DetCond}

Suppose that 
\begin{align*}
	\norm{\calP_{\calS} \calA^*\calA\calP_\calS - \calP_\calS}_{F\to F} \leq \delta \quad \max_{(i,j) \in \calS^c} \norm{\calP_S \calA^*\calA^i_j}_{2\to F} \leq \beta
\end{align*} 
for some $\delta \in [0, 1)$ and $\beta \geq 0$. Also suppose that there exists a matrix tuple $\calY = \calA^* \upsilon$ (an \emph{approximate dual certificate}) with
\begin{align*}
	\norm{\calP_\calS(\calY-\widehat{\calZ}_0)}\leq \eta, \quad \norm{\calY_{\calS^c}}_{\infty,2} \leq \theta \text{ and } \norm{\upsilon}_2 \leq \tau \sqrt{s}.
\end{align*}
 If $\rho= \theta + \tfrac{\eta \beta}{1-\delta} <1$, any solution $\cal{Z}^*$ of $\calP_{1,2}^\sigma$ obeys
\begin{align*}
	\norm{\calZ^*-\calZ_0}_F \leq C_1 \norm{\calP_{\calS^c}\calZ_0}_{1,2} + (C_2 + C_3 \sqrt{s})\sigma.
\end{align*}
The constants are given by
\begin{align*}
	C_1 = \frac{2}{1-\rho} + \frac{2\beta}{(1-\rho)(1-\delta)} \quad C_2 = \frac{2\mu \eta }{1-\rho} + \frac{2\beta \mu \eta}{(1-\rho)(1-\delta)} + 2\mu \quad C_3 = \frac{2\tau}{1-\rho} + \frac{2\beta \tau}{(1-\rho)(1-\delta)},
\end{align*}
where we defined $\mu= \tfrac{\sqrt{1+\delta}}{1-\delta}$.
\end{lem}

	The road ahead is now clear: what we need to do is to prove that with the assumptions we have made, the parameters $\delta$ and $\beta$ will probably be small. We will also have to construct an approximate dual certificate $\calY$ with $\eta$, $\theta$ and $\tau$ as small as possible. Most of these proofs in the following will follow this general structure: 
	\begin{enumerate}
 \item Write the stochastic matrix at hand as a sum of random, centered matrices. 
 \item Estimate the parameters Orlicz-Norms (see Section \ref{sec:Proofs}) and variances of the matrices involved. 
 \item Apply results about norm concentration of sums of independent matrices (see Theorem \ref{th:SwissArmyKnife} and Corollary \ref{cor:SwissNonRectangular}).
\end{enumerate}

Bounding $\delta$ and $\beta$ is particularly natural using the above strategy. The following results hold.
\begin{lem} \label{lem:RIP}
	Under the assumptions $(a)$ to $(d)$, we have
		\begin{align} \label{eq:RIP}
		\norm{\calP_\calS \calA^* \calA\calP_\calS - \calP_\calS}_{F \to F} \leq \frac{1}{4}
	\end{align}
	with a probability larger than $1-\epsilon$.
\end{lem}


\begin{lem}
	\label{lem:OffDiagonal}  Suppose that \eqref{eq:RIP} and assumption $(d)$ is true. Then
	\begin{align*}
		\max_{(i,j)\in \calS^c} \norm{\calP_S \calA^* \calA^i_j}_{2 \to F} \leq \frac{5}{4}
	\end{align*}
	with a probability greater than $1-\epsilon$.
\end{lem}

It now only remains to construct the dual certificate $\calY$. Although Lemma \ref{lem:DetCond} only calls for an \emph{approximate} dual certificate, we will construct an exact one, just as in the paper (Ling and Strohmer 2015 \cite{ling2015sparseDeconv}). This does not only yield good results, but has the main technical advantage that it avoids using the so-called \emph{golfing scheme} (Gross 2011 \cite{gross2011golfing}), which would need further assumptions on our frames $(b^i_\ell)$. It should be noted that this comes at the cost of us not being able to consider more structured $(a^i_\ell)$-constructions. For more details on this issue, see (Ling and Strohmer 2015 \cite{strohmer2015demixing,ling2015sparseDeconv}.)

So, the certificate we will use is defined as follows:
\begin{align}
	\upsilon = \calA_\calS (\calA_\calS^*\calA_\calS)^{-1} \widehat{\calZ}_0, \quad \calY = \calA^*\upsilon , \label{eq:dualCert}
\end{align}
where we introduced the short-hand $\calA_\calS = \calA \calP_\calS$ The following lemma shows that it with very high probability will have the properties we need.

\begin{lem} \label{lem:dualCert} Under assumptions (a)-(d) and additionally that \eqref{eq:RIP} is true, $\upsilon$ and $\calY$ defined in \eqref{eq:dualCert} obeys
\begin{align*}
	\calY_\calS = \calP_\calS \widehat{\calZ}_0 , \quad \norm{\calY_{\calS^c}}_{\infty, 2} \leq \frac{1}{2} \quad \norm{\upsilon} \leq \tfrac{2\sqrt{5}}{3} \sqrt{s}
\end{align*}
with a probability larger than $1-\epsilon$.
\end{lem}

 With the above results at hand, the main result is easily deduced.
 
 \begin{proof}[Proof of Theorem \ref{th:mainResult}]
 	We want to apply Lemma \ref{lem:DetCond}. Lemma \ref{lem:RIP} together with assumptions $(a)$ to $(d)$ secure that with a failure probability smaller than $\epsilon$, \eqref{eq:RIP} holds, i.e. that $\delta \leq \frac{1}{4}$. \eqref{eq:RIP} also makes Lemma \ref{lem:OffDiagonal} appliciable, which proves that $\beta \leq \tfrac{5}{4}$ with a probability of failure smaller than $\epsilon$. It also makes the Lemma \ref{lem:dualCert} about the dual certificate appliciable, which implies that $\eta=0$, $\tau \leq \tfrac{2\sqrt{5}}{3}$ and $\theta \leq \tfrac{1}{2}$ with a probability of failure smaller than $\epsilon$. 
 	
 	All in all, $\rho = \theta + \tfrac{\eta \beta}{1-\delta} \leq \frac{1}{2}<1$ with a probability larger than $1-3\epsilon$, which is what was to be proven. The corresponding bounds on $C_1$, $C_2$ and $C_3$ are
 	\begin{align*}
 		C_1 = \tfrac{32}{3}, \quad C_2 = \tfrac{4}{\sqrt{5}}, \quad C_3 = \tfrac{96\sqrt{5}}{9}.
 	\end{align*}
 \end{proof}

\section{Proofs} \label{sec:Proofs}

In this section, we present all of the technical details omitted above.

\subsection{Lemma \ref{lem:DetCond}}
Let us begin by performing a (relatively straight-forward) calculation of the subdifferential of $\norm{\cdot}_{1,2}$.
\begin{lem} \label{lem:subDiff12}
	Let $\calZ_0 \in \bigoplus_{i \in [r]} \C^{k_i,n_i}$ be supported on the set $\calS$. Then the subdifferential of $\norm{\cdot}_{1,2}$ at $\calZ_0$, i.e. the set of $\xi \in \bigoplus_{i \in [r]} \C^{k_i,n_i}$ with the property
	\begin{align*}
	 \forall \calH \in \bigoplus_{i \in [r]} \C^{k_i,n_i} : \ \norm{\calZ + \calH}_{1,2} \geq \norm{\calZ}_{1,2} + \text{\emph{Re}}(\sprod{\calH, \xi}),
	\end{align*}
	is given by the direct sum of the individual sub-differentials $\partial_{Z_i}(\norm{\cdot}_{1,2})$, $i=1 , \dots, r$, where $\partial_{Z_i}(\norm{\cdot}_{1,2})$  is given by
	\begin{align} \label{eq:subDiffOneMatrix}
		 \set{ V \in \C^{k_i,n_i} \vert V(j) = \tfrac{Z_i(j)}{\norm{Z_i(j)}_2}, j \in S_i , \norm{V(j)}_2 \leq 1, j \notin S_i}.
	\end{align}
\end{lem}
\begin{proof}
	By testing with $\calH$'s with only one $H_i \neq 0$, we see that $\partial_{\calZ_0}\norm{\cdot}_{1,2}$ has the claimed direct sum structure. To calculate $\partial_{Z_i}(\norm{\cdot}_{1,2})$, we need to characterize the matrices $V$ with
	\begin{align} \label{eq:subDiffProp}
		\sum_{j \in [n_i]} \norm{Z_i(j) + H(j)}_2 \geq \sum_{j \in S_i} \norm{Z_i(j)}_2 + \text{Re}(\sprod{V(j), H(j)}) + \sum_{j \notin S_i} \text{ Re}(\sprod{V(j),H(j)})
\end{align}	
for all $H \in \C^{k_i,n_i}$. First, it is easy to see that a matrix tuple in \eqref{eq:subDiffOneMatrix} has this property. To see that $\partial_{Z_i} \norm{\cdot}_{1,2}$ is contained in the set \eqref{eq:subDiffOneMatrix}, begin by testing \eqref{eq:subDiffProp} with arbitrary matrices $H$ supported on single columns with indices in $S_i^c$. The resulting inequality implies that the corresponding columns have norm at most $1$. Similarly for $j \in S_i$, by testing with the matrix with $j$:th column $\pm Z_i(j)$ and $0$ else, we see that $\text{Re}(\sprod{Z_i(j), V(j)}) =1$, i.e. $V(j) = \tfrac{Z_i(j)}{\norm{Z_i(j)}_2} + u_j$ with $\text{Re}(u_j, Z_i(j))=0$. To see that $u_j=0$, test with the matrix having $j$:th column  $\tau u_j$ for $\tau>0$ and zero else. Since $\norm{v+w}_2^2 = \norm{v}_2^2 + \norm{w}_2^2 + 2 \text{Re}(\sprod{v,w})$, this implies
\begin{align*}
	\sqrt{\norm{Z_i(j)}_2^2+ \tau^2 \norm{u_j}_2^2} \geq \norm{Z_i(j)}_2 + \tau \norm{u_j}_2^2 \Rightarrow \norm{u_j}_2^2 \leq \frac{ \sqrt{\norm{Z_i(j)}_2^2+ \tau^2 \norm{u_j}_2^2}- \norm{Z_i(j)}_2}{\tau} , \tau >0.
\end{align*} 
By letting $\tau \to 0$, we obtain $\norm{u_j}_2=0$.
\end{proof}

We can now prove Lemma \ref{lem:DetCond}, using the same ideas as in the proof of its $\ell_1$-counterpart (Foucart and Rauhut 2013 \cite[Th. 4.33]{MathIntroToCS}).

\begin{proof}
	Let us denote $\calH= \calZ^*-\calZ_0$.
	Then we have due to the triangle inequality and Lemma \ref{lem:subDiff12} 
	\begin{align*}
		\norm{\calZ^*}_{1,2} = \norm{\calZ_0 + \calH}_{1,2} \geq \norm{\calP_\calS \calZ_0 + \calH}_{1,2}- \norm{\calP_{\calS^c}\calZ_{0}} \geq \norm{\calP_\calS \calZ_0} + \text{Re}\left(\sprod{\calP_\calS \widehat{\calZ}_0, \calH_\calS}\right) + \sprod{\xi_{\calS^c}, \calH_\calS} - \norm{\calP_{\calS^c}\calZ_{0}} 
	\end{align*}
	for every $\xi$ with $\norm{\xi}_{\infty,2} \leq 1$.  Now since $\calZ_0$ obeys the constraint of $\calP_{1,2}^\sigma$, there must be $\norm{\calZ_0}_{1,2} \geq\norm{\calZ^*}_{1,2}$. Using this, the above inequality, and choosing $\xi$ appropriately, we obtain
	\begin{align} \label{eq:HoffSest}
		\norm{\calH_{\calS^c}}_{1,2} \leq \norm{\calZ_0}_{1,2} - \norm{\calP_\calS \calZ_0}_{1,2} + \norm{\calP_{\calS^c} \calZ_0}_{1,2} + \abss{\sprod{\calP_\calS \widehat{\calZ}_0,\calH_\calS}} \leq 2\norm{\calP_{\calS^c} \calZ_0}_{1,2} + \abss{\sprod{\calP\widehat{\calZ}_0,\calH_\calS}}.
	\end{align}
	Due to the first property of $\calY$, we have
	\begin{align} \label{eq:ScalProdEst}
		\abss{\sprod{\calP_\calS \widehat{\calZ}_0, \calH_\calS}} \leq \abss{\sprod{\calP_S(\widehat{\calZ}_0-\calY), \calH_\calS}} + \abss{\sprod{\calY_\calS, \calH_\calS}} \leq \eta \norm{\calH_\calS}_F + \abss{\sprod{\calY, \calH}} + \abss{\sprod{\calY_{\calS^c},\calH_{\calS^c}} }.
	\end{align}
	Due to $\norm{\calP_{\calS} \calA^*\calA\calP_\calS - \calP_\calS}_{F\to F} \leq \delta$, we furthermore have
	\begin{align} \label{eq:ineqDeterm}
		\norm{\calH_\calS}_F \leq \frac{1}{1-\delta} \norm{\calP_S \calA^* \calA \calP_\calS \calH}_F \leq \frac{1}{1-\delta}\norm{\calP_S \calA^* \calA \calH}_F + \frac{1}{1-\delta}\norm{\calP_S \calA^* \calA \calP_{\calS^c}\calH}_F.
	\end{align}
	Now we estimate both of these two terms separately, starting with the second one. We have, due to
	\begin{align*}
		\calA(\calH_{\calS^c}) = \sum_{(i,j)\in \calS^c} \calA^{(i,j)}(H_i(j)),
	\end{align*}
	that
	\begin{align} \label{eq:PsAAHSc}
		\norm{\calP_S \calA^* \calA (\calH_{\calS^c})}_F \leq  \sum_{(i,j)\in \calS^c} \norm{\calP_S \calA^*\calA^{(i,j)}(H_i(j))}_F \leq \beta \sum_{(i,j)\in \calS^c} \norm{H_i(j)}_2 \leq \beta \norm{\calH_{\calS^c}}_{1,2},
	\end{align}
 where we in the second to last step used that $\norm{\calP_S \calA^*\calA^i_j}_{2\to F} \leq \beta$ for all $(i,j) \in \calS^c$. Now for the first term in \eqref{eq:ineqDeterm}. Since $\norm{\calP_\calS \calA^*}_{2\to F}^2 = \norm{\calP_\calS \calA^* \calA \calP_S}_{F \to F} \leq 1 +\delta$, we have 
 \begin{align} \label{eq:PsAHSc}
 	\norm{\calP_\calS \calA^*\calA(\calH)}_F \leq \sqrt{1+\delta} \norm{\calA (\calH)}_2 \leq 2 \sqrt{1+\delta}\sigma,
 \end{align}
 where the last estimate follows from the constraint of $\calP_{1,2}^\sigma$:
 \begin{align*}
 	\norm{\calA(\calH)}_2 \leq \norm{\calA(\calZ_0)-y}_2 + \norm{y-\calA(\calZ^*)}_2 \leq 2 \sigma.
 \end{align*}
Combining \eqref{eq:ineqDeterm} with \eqref{eq:PsAAHSc} and \eqref{eq:PsAHSc},we obtain
\begin{align} \label{eq:ineqDeterm2}
\norm{\calH_\calS}_F  \leq \frac{1}{1-\delta}\left(2 \sqrt{1+\delta} \sigma + \beta \norm{\calH_{\calS^c}}_{1,2}\right)
\end{align}

	We may furthermore deduce from the fact that $\calY = \calA^*\upsilon$ and $\norm{\upsilon}_2 \leq 2 \tau \sqrt{s}$
\begin{align} \label{eq:scalProd2}
	\abss{\sprod{\calY, \calH}} = \abss{\sprod{\upsilon, \calA(\calH)}} \leq \norm{\upsilon}_2 \norm{\calA(\calH)}_2 \leq 2 \tau \sqrt{s} \sigma
\end{align}
Finally, due to the second assumption on $\calY$, $\abss{\sprod{\calY_{\calS^c}, \calH_{\calS^c}}} \leq \norm{\calY_{\calS^c}}_{\infty,2} \norm{\calH_{\calS^c}}_{1,2} \leq \theta \norm{\calH_{\calS^c}}_{1,2}
$. Putting this estimate together with \eqref{eq:ineqDeterm2} and \eqref{eq:scalProd2} into \eqref{eq:ScalProdEst} yields
\begin{align*}
	\abss{\sprod{\calP_\calS \widehat{\calZ}_0, \calH_\calS}} \leq \frac{2\eta}{1-\delta} \sqrt{1+\delta}\sigma + \frac{\eta\beta}{1-\delta}\norm{\calH_{\calS^c}}_{1,2} + 2 \tau \sigma \sqrt{s} + \theta \norm{\calH_{\calS^c}}_{1,2}
\end{align*}
Which, put into \eqref{eq:HoffSest} and identifying the expressions defined in the statement of the theorem, reads
\begin{align*}
	\norm{\calH_{\calS^c}}_{1,2} \leq  2\norm{\calP_{\calS^c} \calZ_0}_{1,2} +  2\eta \mu\sigma + 2 \tau \sigma \sqrt{s} + \rho \norm{\calH_{\calS^c}}_{1,2}.
\end{align*}
Which together with \eqref{eq:ineqDeterm2} implies
\begin{align*}
	\norm{\calH}_F &\leq \norm{\calH_\calS}_F + \norm{\calH_{\calS^c}}_F \leq 2 \mu \sigma + \left(1 + \frac{\beta}{1-\delta}\right)\norm{\calH_{\calS^c}}_{1,2} \\
	&\leq 2 \mu \sigma + \left(1 + \frac{\beta}{1-\delta}\right)\cdot \frac{1}{1-\rho} \cdot \left(2\norm{\calP_{\calS^c} \calZ_0}_{1,2} +  2\eta \mu\sigma + 2 \tau \sigma \sqrt{s} \right), 
\end{align*}
which is exactly what we aimed to prove.
	\end{proof}
	
	\subsection{A Technical Tool from Random Matrix Theory.}
	Just as in the paper (Ling and Strohmer \cite{ling2015sparseDeconv}), the main technical tool is a version of the Matrix Berstein inequality.  We use the one from  (Koltchinskii 2013 \cite{ koltchinskii2013}). It makes use of the $1$-Orlicz-Norm of a random matrix:
	\begin{align*}
		\norm{\Psi}_{\psi_1} = \inf_{u \geq 0} \erw{\exp(\norm{\Psi/u}_{2\to 2})} \leq 2.
	\end{align*}
	It is possible to prove (Vershynin 2012 \cite[Lem 5.5]{RandomMatrices2012Vershinyn}) that the Orlicz norm is equivalent to 
	\begin{align*}
		\norm{\Psi}_{\psi_1}:= \sup_{p\geq 1} p^{-1} \erw{\norm{\Psi}_{2 \to 2}^p}^{-\tfrac{1}{p}}.
	\end{align*}
	This makes it clear that it is reasonable to define the $\psi_q$-norm of a random variable $X$ for $q \geq 1$ through
	\begin{align*}
		\norm{X}_{\psi_q} \simeq \sup_{p\geq 1} p^{-\tfrac{1}{q}} \erw{\abs{X}^p}^{-\tfrac{1}{p}}.
	\end{align*}
	We then have $\norm{\abs{X}^q}_{\psi_1} \leq q\norm{X}_{\psi_q}^q$, since
	\begin{align*}
		\norm{\abs{X}^q}_{\psi_1} = \sup_{p\geq 1} p^{-1} \erw{\abs{X}^{qp}}^{-\tfrac{1}{p}} =  \sup_{p \geq 1} \left( (pq)^{-\tfrac{1}{q}}q^{\tfrac{1}{q}} \erw{\abs{X}^{qp}}^{-\tfrac{1}{pq}}\right)^q =\sup_{p \geq q} \left( p^{-\tfrac{1}{q}}q^{\tfrac{1}{q}} \erw{\abs{X}^{p}}^{-\tfrac{1}{p}}\right)^ q \leq q\norm{X}_{\psi_q}^q. 
\end{align*}	
It is also clear that if $X$ and $Y$ are independent, we have 
\begin{align}
	\norm{X \cdot Y}_{\psi_1} = \sup_{p \geq 1} p^{-1}\erw{ \abs{X \cdot Y}^p}^{-\tfrac{1}{p}} = \sup_{p \geq 1} p^{-\tfrac{1}{2}}\erw{ \abs{X}^p}^{-\tfrac{1}{p}} \cdot \sup_{p \geq 1} p^{-\tfrac{1}{2}}\erw{ \abs{Y}^p}^{-\tfrac{1}{p}} =  \norm{X}_{\psi_2} \norm{Y}_{\psi_2} \label{eq:OrliczIndep}
\end{align}

 For a vector $g =(g(1), \dots g(d)) \in \C^d$ with independent Gaussian entries, with variances $\mathfrak{s}_i^2$, we have (see (Vershynin 2012 \cite[Example 5.8.1, Lemma 5.9]{RandomMatrices2012Vershinyn} ))
\begin{align}
		\norm{g}_{\psi_2}^2 \leqsim \sum_{i\in [d]} \norm{g(i)}_{\psi_2}^2  \leqsim\sum_{i\in [d]} \mathfrak{s}_i^2 . \label{eq:GaussOrlicz}
\end{align} 

\begin{rem}
Note that  in this paper, $\norm{g}_{\psi_2}$ is for a random vector $g \in \C^d$, despite of the notional similarity, \emph{not} the subgaussian norm defined in (Vershynin 2012 \cite{RandomMatrices2012Vershinyn}). Instead, we view it as a linear map $\C^d \to \R$ and use the definition from above, i.e. $\norm{g}_{\psi_q} = \norm{\norm{g}_2}_{\psi_q}$.
\end{rem}	
	With this terminology at hand, we may formulate the theorem. To keep things simple, we omit some details, in particular regarding the values of the appearing constants. These can be found in the referenced source.
	\begin{theo} \label{th:SwissArmyKnife} (Simplified version of Koltchinskii 2013 \cite[Theorem 4, Corollary 1]{koltchinskii2013})\footnote{In an earlier version of this article, a similar bound, which is only true for \emph{identically distributed} matrices, was erroneously used.} Consider a finite sequence $(\Psi_\ell)_{\ell \in [q]}$ of independent, centered and self-adjoint random matrices with dimension $M \times M$, with $R:= \max_{1 \leq 1\leq q}\norm{\Psi_\ell}_{\psi_1}<\infty$. Let ${\bf S}$ denote the matrix
	\begin{align*}
		{\bf S} = \sum_{\ell=1}^q \Psi_\ell
	\end{align*}
	and define
	\begin{align*}
		\sigma^2 := \max \left(\norm{\sum_{\ell \in [q]} \erw{\Psi_\ell \Psi_\ell^*}}_{2 \to 2} , \norm{\sum_{\ell \in [q]} \erw{\Psi_\ell^* \Psi_\ell}}_{2 \to 2}\right).
	\end{align*}
	Then we can bound
	\begin{align*}
	\prb{\norm{{\bf S}}_{2 \to 2} \geq t } \leq P_{M, \sigma^2, R, q}(t),
	\end{align*}
	where $P_{M,\sigma^2, R, q}(t)$ has the following property:
	 There exist constants $C_0$, $C_1$ and $C_2$ such that for every $t$ satisfying $0< t R \log\left(1+ \tfrac{C_2R^2q}{\sigma^2}\right)< C_1 \sigma^2$,
	we have
	\begin{align*}
		 P_{M, \sigma^2, R, q}(t) \leq 2 M \exp\left( - \frac{1}{C_0}\frac{t^2}{\sigma^2 +  R\log\left(1+ \tfrac{C_2R^2q}{\sigma^2}\right)t}\right).
	\end{align*}
	For $t$ with $ tR\log\left(1+ \tfrac{C_2R^2q}{\sigma^2}\right) \geq \sigma^2$, we instead have
	\begin{align*}
		P_{M, \sigma^2, R, q}(t) \leq 2 M \exp\left( - \frac{1}{C_0}\frac{t}{ R\log\left(1+ \tfrac{C_2R^2q}{\sigma^2}\right)}\right)
\end{align*}	 
	\end{theo}

As was pointed out in (Tropp 2012 \cite{Tropp2012userfriendly}), a theorem like the previous one immediately implies a corresponding statement for non-square (and also square but non-self-adjoint) matrices. Let us state and prove this assertion.
\begin{cor} \label{cor:SwissNonRectangular}
	Consider a sequence $(\Psi_\ell)_{\ell \in [q]}$ of independent centered random matrices with dimension $M \times N$. Adopting the notation of the previous theorem, we then have 
	\begin{align*}
	\prb{\norm{{\bf S}}_{2 \to 2} \geq t } \leq  P_{M+N, \sigma^2, R, q}(t).
	\end{align*}

\end{cor}
\begin{proof}[Sketch of Proof]
	(The idea is from (Tropp 2012 \cite{Tropp2012userfriendly})). For a matrix $M \in \C^{M, N}$, define the \emph{dialation} $\scrS(M) \in \C^{M+N, M+N}$ through
	\begin{align*}
		\scrS(M) = \begin{bmatrix}
		 0 & M\\
		 M^* & 0
		\end{bmatrix}.
	\end{align*}
	Then $\norm{\scrS(M)}_{2 \to 2} = \norm{M}_{2 \to 2}$. Consequently, $\scrS(\Psi_\ell)$ is a sequence of independent and centered self-adjoint $(M+N)\times (M+N)$-matrices with the same $R$ and $\sigma^2$-parameters as $(\Psi_\ell)$. Hence, the statement follows immediately from the previous theorem.
\end{proof}

We will often use the above theorem to derive bounds on the number of measurements needed for the probability that $\norm{{\bf S}}_{2\to 2}$ to be small for some random matrix ${\bf S}$. When doing this, the following observation be very convenient: 
Let $t>0$. Suppose that we have secured a bound of the form 
\begin{align*}
	\prob{Event} \leq QP_{M,\sigma^2,R,q}(t)
\end{align*}
and that
\begin{align}
	\sigma^2 +  Rt \log\left(1+ \tfrac{C_2R^2q}{\sigma^2}\right) \leqsim t^2\log\left(\tfrac{QM}{\epsilon}\right)^{-1}, \quad R \log\left(1+ \tfrac{C_2R^2q}{\sigma^2}\right) \leqsim t\log\left(\tfrac{QM}{\epsilon}\right)^{-1} \label{eq:ProbAsymptotics}
\end{align}
By applying the bounds provided in Theorem \ref{th:SwissArmyKnife} for large and small $t$ separately, we can then conclude that $\prob{Event}<\epsilon$.

\subsection{Bounding the Parameters $\beta$ and $\delta$.}
With the two results presented in the last section in our toolbox, it is possible to bound the parameters $\beta$ and $\delta$ with high probability. 

For a start, note that there is no fundamental difference between dealing with linear maps on the space of matrix tuples equipped with the Frobenius norm and matrices defined on $\C^K$, for an appropriate $K$, equipped with the $\ell_2$-norm. We will from now on never comment on this subtlety and instead apply Theorem \ref{th:SwissArmyKnife} and Corollary \ref{cor:SwissNonRectangular} without explicitely re-interpreting the linear maps on matrix spaces to linear maps on an high-dimensional $\C^K$.

Let us now 
	calculate $\calA^*$. We have for $\calZ \in \bigoplus_{i \in [r]} \C^{k_i,n_i}$ and $p \in \C^q$
\begin{align*}
	\sprod{\calA(\calZ),p} = \sum_{\ell \in [q]} \sum_{i \in [r]} \overline{\sprod{b^i_\ell, Z_i a^i_\ell}} p_\ell = \sum_{i \in [r]} \sprod{Z_i,p_\ell b^i_\ell(a^i_\ell)^*} = \sprod{\calZ, \calA^* p},
\end{align*}
and hence
\begin{align*}
	(\calA^*p)_i = \sum_{\ell \in [q]} p_\ell b^i_\ell(a^i_\ell)^*.
\end{align*}
Consequently,
\begin{align*}
	(\calA^* \calA (\calZ))_i =  \sum_{\ell \in [q]} \sum_{\kappa \in [r]} \sprod{b^\kappa_\ell, Z_\kappa a^\kappa_\ell} b^i_\ell (a^i_\ell)^* = \sum_{\ell \in [q]}  \sum_{\kappa \in [r]} b^i_\ell(b^\kappa_\ell)^* Z_\kappa a^\kappa_\ell(a^i_\ell)^*
\end{align*}

We can now provide the proofs, starting with Lemma \ref{lem:RIP}.
\begin{proof}[Proof of Lemma \ref{lem:RIP}]
	We have
	\begin{align*}
		\left(\calP_\calS \calA^*\calA\calP_\calS(\calZ)\right)_i = \sum_{\ell \in [q]} \sum_{\kappa \in [r]} \sprod{b^\kappa_\ell, Z_\kappa (a^\kappa_\ell)_{S_\kappa}} b^i_\ell (a^i_\ell)_{S_i}^*=\sum_{\ell \in [q]} \sum_{\kappa \in [r]}  b^i_\ell (b^\kappa_\ell)^* Z_\kappa \alpha^\kappa_\ell (\alpha^i_\ell)^*,
	\end{align*} 
	where we for $i, \ell$ defined the variables $\alpha^i_\ell = (a^i_\ell)_{S_i}$, which again are independent and Gaussians in their respective spaces, since they are projections of independent Gaussians. Due to further basic properties of Gaussians, we have $\erw{\alpha^\kappa_\ell (\alpha^i_\ell)^*} = \delta_{\kappa i} \id_{S_i}$.  Due to assumption (b), we furthermore have
	\begin{align*}
		\calZ = \left(\sum_{\ell \in [q]} b^i_\ell (b_\ell^i)^* Z_i\right)_{i \in [r]},
	\end{align*}
	and hence
	\begin{align*}
		\left(\calP_\calS \calA^* \calA\calP_\calS - \calP_\calS\right)= \sum_{\ell \in [q]} \Psi_\ell- \erw{\Psi_\ell},
	\end{align*}
	where we defined $\Psi_\ell: \bigoplus_{i \in [r]} \C^{k_i,n_i} \to \bigoplus_{i \in [r]} \C^{k_i,n_i}$ through
	\begin{align*}
		\Psi_\ell(Z)_i = \sum_{\kappa \in [r]} b^i_\ell (b^\kappa_\ell)^* Z_\kappa \alpha^\kappa_\ell (\alpha^i_\ell)^*
	\end{align*}
	The random variables $\Psi_\ell - \erw{\Psi_\ell}$, $\ell =1, \dots, q$, are independent and, of course, centered. In order to apply Theorem \ref{th:SwissArmyKnife}, we need to estimate the $\psi_1$-norms of them. Towards this, let us begin by calculating $\norm{\Psi_\ell(\calZ)}_F$ for a fixed $\calZ$. We have
	\begin{align*}
		\norm{\Psi_\ell(\calZ)}_F^2 &= \sum_{i \in [r]} \big\Vert \sum_{\kappa \in [r]} \sprod{b_\ell^\kappa,Z_\kappa \alpha_\ell^\kappa} b_\ell^i (\alpha_\ell^i)^* \big\Vert_{F}^2 = \sum_{i \in [r]} \norm{b_\ell^i(\alpha^i_\ell)^*}_F^2 \bigg\vert \sum_{k \in [r]} \sprod{b_\ell^\kappa, Z_\kappa \alpha_\ell^\kappa} \bigg\vert^2 \\
		&\leq \sum_{i \in [r]} \norm{b_\ell^i}_2^2\norm{\alpha^i_\ell}_2^2 \left( \sum_{k \in [r]} \norm{b_\ell^\kappa}_2 \norm{ Z_\kappa}_F \norm{\alpha_\ell^\kappa}_2\right)^2 \leq \sum_{i \in [r]} \norm{b_\ell^i}_2^2\norm{\alpha^i_\ell}_2^2\sum_{j \in [r]} \norm{b_\ell^j}_2^2  \norm{\alpha_\ell^j}_2^2 \sum_{\kappa \in [r]}\norm{ Z_\kappa}_F^2
	\end{align*}
	We used Cauchy-Schwarz, $\norm{Av}_2 \leq \norm{A}_F \norm{v}$ and $\norm{uv^*}_F = \norm{u}_2\norm{v}_2$. Hence, $\norm{\Psi_\ell(\calZ)}_F \leq \sum_{i \in [r]} \norm{b_\ell^i}_2^2\norm{\alpha^i_\ell}_2^2 \cdot \norm{\calZ}_F$, and consequently
	\begin{align*}
		\norm{\Psi_\ell}_{F \to F} \leq \sum_{i \in [r]} \norm{b_\ell^i}_2^2\norm{\alpha^i_\ell}_2^2.
	\end{align*}
	This is an expression which obeys
	\begin{align*}
		\norm{\norm{\Psi_\ell}_{F \to F}}_{\psi_1} \leq 2\sum_{i \in [r]} \norm{b_\ell^i}_2^2\norm{\alpha^i_\ell}_{\psi_2}^2 \leqsim \frac{\mu_{+}^2}{q}\sum_{i \in [r]}  k_is_i .
	\end{align*}
	We used $\norm{X^2}_{\psi_1} \leq 2 \norm{X}_{\psi_2}^2$,   \eqref{eq:GaussOrlicz}, and assumption $(c)$. Note that we can use the same (asymptotic) estimate for the $\psi_1$-norm of $\Psi_\ell - \erw{\Psi_\ell}$, as was pointed out in (Vershynin 2012 \cite{RandomMatrices2012Vershinyn}). We have hence managed to bound the $R$-parameter in Theorem \ref{th:SwissArmyKnife}.
	
	Let us move on to the $\sigma^2$-parameter. $\Psi_\ell$, and therefore also $\Psi_\ell - \erw{\Psi_\ell}$, is self-adjoint, since
	\begin{align*}
		\sprod{\Psi_\ell(\calZ), \calY} = \sum_{i \in [r]} \sprod{ \sum_{\kappa \in [r]} b^i_\ell(b^\kappa_\ell)^* Z_\kappa \alpha_\ell^\kappa(\alpha_\ell^i)^*, Y_i} = \sum_{\kappa \in [r]}  \sprod{Z_\kappa,\sum_{i \in [r]} b^\kappa_\ell(b^i_\ell)^*Y_i\alpha_\ell^i(\alpha_\ell^\kappa)^*} = \sprod{\calZ, \Psi_\ell(\calY)}.
	\end{align*}
	Therefore, 
	\begin{align*}
		\erw{(\Psi_\ell - \erw{\Psi_\ell})^*(\Psi_\ell - \erw{\Psi_\ell}) }=\erw{(\Psi_\ell - \erw{\Psi_\ell})^2 }= \erw{\Psi_\ell^2} - \erw{\Psi_\ell}^2.
	\end{align*}
	$\erw{\Psi_\ell}^2$ is given by $\erw{\Psi_\ell}^2(\calZ)_i = \norm{b_\ell^i}_2^2 b_\ell^i (b_\ell^i)^*Z_i$ and
	\begin{align*}
		\Psi_\ell^2(\calZ)_i = \sum_{j \in [r]} b^i_\ell (b^j_\ell)^* \Psi_\ell(\calZ)_j \alpha^j_\ell (\alpha^i_\ell)^* =  \sum_{j \in [r]} \sum_{\kappa \in [r]}  b^i_\ell \norm{b^j_\ell}_2^2 (b_\ell^\kappa)^* Z_\kappa \alpha_\ell^\kappa \norm{\alpha_\ell^j}_2^2 (\alpha^i_\ell)^*.
	\end{align*}
	Lemma \ref{lem:Gauss} (which is yet to be proven) reads
	\begin{align*}
		\erw{\alpha_\ell^\kappa \norm{\alpha_\ell^j}_2^2 (\alpha^i_\ell)^*} = \begin{cases} (s_i +2) \id_{S_i} & \quad i=j=\kappa \\
		s_j \id_{S_i} &\quad i=\kappa \neq j \\
		0  & \quad \text{ else. } \end{cases} 
\end{align*}		
Consequently,
\begin{align*}
	\erw{\Psi_\ell^2(\calZ)_i} = (s_i+2) \norm{b_\ell^i}_2^2 b_\ell^i (b_\ell^i)^*Z_i + \sum_{j \neq i} s_j \norm{b^j_\ell}_2^2 b^i_\ell (b_\ell^i)^*Z_i.
\end{align*}
This implies
\begin{align*}
	\left(\erw{\Psi_\ell^2} -\erw{\Psi_\ell}^2\right)_i = \left((s_i+1) \norm{b_\ell^i}_2^2 + \sum_{j \neq i} s_j \norm{b^j_\ell}_2^2\right) b^i_\ell (b_\ell^i)^* \asymp  \frac{\mu^2}{q} \left(\sum_{j \in [r]} s_j k_j \right) b_\ell^i(b_\ell^i)^*.
\end{align*}
We used assumption $(c)$ at the end. Here, $\mu$ is meant to be understood as $\muplus$ in the upper bound and $\muminus$ in the lower bound. Summing over $\ell \in [q]$ and utilizing assumption $(b)$, we arrive at
\begin{align*}
	\sum_{\ell \in [q]} \erw{\Psi_\ell^2}- \erw{\Psi_\ell}^2  \asymp \frac{\mu^2}{q}\big( \sum_{i \in [r]} s_i k_i \big)
\end{align*}
I.e. $R \leqsim \frac{\muplus^2}{q} \big( \sum_{i \in [r]} s_i k_i \big)$ and $\sigma^2 \asymp \tfrac{\mu^2}{q}\big( \sum_{i \in [r]} s_i k_i \big)$. Towards applying Theorem \ref{th:SwissArmyKnife}, we note that these bounds together with assumption $(d)$ imply that
\begin{align*}
	\sigma^2 + \tfrac{1}{2} R \log\left(1 + \tfrac{C_2R^2q}{\sigma^2}\right) &\leqsim \tfrac{\muplus^2}{q}\big( \sum_{i \in [r]} s_i k_i \big)\left( 1 + \log \left(1+\tfrac{C_2\mu_+^4\sum_{i \in [r]} s_i k_i}{\muminus^2} \right)\right) \leqsim \frac{1}{4}\log\left(\tfrac{\sum_{i\in [r]}s_ik_i}{\epsilon}\right)^{-1} \\
	\tfrac{1}{4}R\log\left(1 + \tfrac{C_2R^2q}{\sigma^2}\right)&\leqsim \tfrac{\muplus^2}{2q} \big( \sum_{i \in [r]} s_i k_i \big) \leqsim \frac{1}{16}\log\left(\tfrac{\sum_{i\in [r]}s_ik_i}{\epsilon}\right)^{-1}.
\end{align*}
We used that $ \sum_{i\in[r]} s_i k_i \leq n k^*$.
We now Theorem \ref{th:SwissArmyKnife} to conclude that
\begin{align*}
	\prb{\normm{ \calP_\calS \calA^*\calA\calP_\calS - \calP_\calS}_{F \to F} > \tfrac{1}{2}} \leq P_{\left(\sum_{i\in [r]} s_ik_i\right), \sigma^2,R,q}\left(\tfrac{1}{2}\right).
\end{align*}
(Note that $\calP_\calS \calA^*\calA\calP_\calS - \calP_\calS$ is defined on the $\left(\sum_{i\in [r]} s_ik_i\right)$-dimensional space $\set{ \calZ \in \bigoplus_{i\in [r]} \calC^{k_i,n_i} \vert \supp \calZ \sse \calS}$). Since we have a bound of the form \eqref{eq:ProbAsymptotics} for $t = \tfrac{1}{4}$, the discussion following Theorem \ref{th:SwissArmyKnife} finishes the proof.
	\end{proof}

It remains to prove the left out lemma.
	
\begin{lem} \label{lem:Gauss}
	Let $\alpha_\ell^\kappa$ be defined as above. We then have
	\begin{align*}
		\erw{\alpha_\ell^\kappa \norm{\alpha_\ell^j}_2^2 (\alpha^i_\ell)^*} = \begin{cases} (s_i +2) \id_{S_i} & \quad i=j=\kappa \\
		s_j \id_{S_i} &\quad i=\kappa \neq j \\
		0  & \quad \text{ else. } \end{cases} 
\end{align*}	
\end{lem}

\begin{proof}
	In the case that $i \neq \kappa$, we have, due to the independence 
	$$\erw{\alpha_\ell^\kappa \norm{\alpha_\ell^j}_2^2 (\alpha^i_\ell)^*} = \erw{\alpha_\ell^\kappa \norm{\alpha_\ell^j}_2^2}\erw{(\alpha^i_\ell)^*} = \erw{\alpha_\ell^\kappa \norm{\alpha_\ell^j}_2^2} \cdot 0 =0.$$

In the case that $i =\kappa$, but $j$ is distinct from $i$, we have, again due to independence
\begin{align*}
	\erw{\alpha_\ell^\kappa \norm{\alpha_\ell^j}_2^2 (\alpha^i_\ell)^*}= \erw{\norm{\alpha_\ell^j}_2^2} \erw{ \alpha_\ell^i(\alpha^i_\ell)^*}= s_j \id_{S_i}.
\end{align*}
For the final case that all three indices are equal, we first note that $\alpha_\ell^i \sim \rho_i \theta_i$, with $\rho_i$, $\theta_i$ independent, $\rho_i \sim \norm{\alpha_\ell^i}$, i.e. is $\chi_{s_i}$-distributed, and $\theta_i$ uniformly distributed over $\sph^{s_i-1}$. Hence
\begin{align*}
	\erw{\alpha_\ell^i \norm{\alpha_\ell^i}_2^2 (\alpha^i_\ell)^*} = \erw{\rho_i^4} \erw{\theta_i \theta_i^*} = s_i(s_i+2) \tfrac{1}{s_i} \id_{S_i}= (s_i+2)\id_{S_i},
\end{align*}
where the second to last equation follows from
\begin{align*}
	\erw{\rho_i^4} = \sum_{\kappa \in [s_i]} \erw{\alpha_\ell^i(\kappa)^4} + \sum_{\kappa \neq \lambda \in [s_i]} \erw{\alpha_\ell^i(\kappa)^2}\erw{\alpha_\ell^i(\lambda)^2} 
= s_i\cdot 3+ s_i(s_i-1)= s_i(s_i+2) \end{align*}
\end{proof}

Now that we have the $\delta$-parameter under control, the $\beta$-parameter is easy to handle.

\begin{proof}[Proof of Lemma \ref{lem:OffDiagonal}]
\eqref{eq:RIP} implies that $\norm{\calP_\calS \calA^*}_{2 \to F} = \sqrt{\norm{\calP_\calS \calA^*\calA \calP_\calS}_{F \to F}} \leq \sqrt{ \tfrac{5}{4}}$, which in turn implies that
\begin{align*}
	\norm{\calP_S \calA^* \calA^i_j}_{2 \to F} \leq \tfrac{\sqrt{5}}{2} \norm{\calA^i_j}_{2\to 2}.
\end{align*}
It is furthermore clear that $\norm{\calA^i_j}_{2 \to 2} = \sqrt{\norm{(\calA_j^i)^*\calA_j^i}_{2 \to 2}}= \sqrt{\norm{\calP_{(i,j)}\calA^*\calA\calP_{(i,j)}}_{F \to F}}$. The latter expression can be dealt with just as the corresponding one in Lemma \ref{eq:RIP} - Theorem \ref{th:SwissArmyKnife} implies for fixed $(i,j)$
\begin{align*}
	\prb{ \norm{\calP_{(i,j)}\calA^*\calA\calP_{(i,j)} - \calP_{(i,j)}}_{F \to F}>  \tfrac{1}{4}} \leq  P_{1,\sigma^2,R,q}(\tfrac{1}{4})
\end{align*}
with $R \leqsim \tfrac{\muplus^2}{q} k_i$ and $\sigma^2 \asymp \tfrac{\mu^2}{q} k_i$. In particular, we have  $\max\left(\sigma^2 + R \log\left( 1+ \tfrac{C_2R^2q}{\sigma^2}\right), R\log\left(1 + \tfrac{C_2R^2q}{\sigma^2}\right)\right) \leqsim \tfrac{1}{16}\log(\tfrac{nk^*}{\epsilon})^{-1}$ under assumption $(d)$, and we way hence conclude that  $P_{1,\sigma^2,R,q}(\tfrac{1}{4}) \leq \frac{\epsilon}{nk^*}$, which together with a union bound over all (less than $nk^*$) pairs $(i,j)$ proves that $\sup_{i,j} \norm{\calA^{i}_j}_{2\to 2} \leq \sqrt{1+ \tfrac{1}{4}}$, and therefore also the theorem.
\end{proof}

\subsection{The Dual Certificate}

Now we prove that the dual certificate defined in \eqref{eq:dualCert} has the properties we need with high probability.

\begin{proof}[Proof of Lemma \ref{lem:dualCert}]
	Let us begin by noting that \eqref{eq:RIP} implies that $\calA_\calS^*\calA_\calS$ is invertible, with $\norm{(\calA_\calS^*\calA_\calS)^{-1}}\leq \tfrac{4}{3}$. Since also $\norm{\calA_\calS}_{F \to 2} \leq \sqrt{\tfrac{5}{4}}$, we have
	\begin{align*}
		\norm{\upsilon} \leq\tfrac{4}{3} \sqrt{\tfrac{5}{4}}\norm{\widehat{\calZ}_0}_F = \tfrac{2\sqrt{5}}{3}  \sqrt{s},
	\end{align*}
	where the last equality is true since $\widehat{\calZ}_0$ is a tuple of matrices with column-sparsities $s_1, \dots s_r$, and each of the non-zero columns are normalized. It is furthermore clear that
	\begin{align*}
		\calY_\calS = \calA_\calS^* \upsilon = \calA_\calS^*\calA_\calS (\calA_\calS^*\calA_\calS)^{-1}
	 \calP_\calS \widehat{\calZ}_0 = \calP_\calS \widehat{\calZ}_0.
	\end{align*}
	Hence, it just remains to estimate the norms of the columns in $\calY$ corresponding to $(i,j) \notin \calS$. Towards this, let us define the matrix tuple
	\begin{align*}
		\widehat{\calY} = (\calA_\calS^*\calA_\calS)^{-1} \widehat{\calZ}_0.
	\end{align*}
	Then, due to the near-isometry property of $\calA_\calS^*\calA_\calS$ and $\norm{\widehat{\calZ}_0}_F =\sqrt{s}$,		$\tfrac{4}{5}\sqrt{s} \leq \norm{\widehat{\calY}}_F \leq \tfrac{4}{3}\sqrt{s}$. Also for any index $(i,j)$
	\begin{align*}
		Y_i(j) = \sum_{\ell \in [q]} \sum_{\kappa \in [r]} \sprod{b^\kappa_\ell, \widehat{Y}_\kappa (a^\kappa_\ell)_{S_\kappa}} b^i_\ell a^i_\ell(j) \sim \sum_{\ell \in [q]} \varphi_\ell,
	\end{align*}
	where we defined $k_i$-dimensional random vectors $\varphi_\ell$ through
	\begin{align*}
		\varphi_\ell = \sum_{\kappa \in [r]} \sprod{b^\kappa_\ell, \widehat{Y}_\kappa \alpha^\kappa_\ell} b^i_\ell \gamma_\ell
	\end{align*}
	with $\gamma_\ell \in \C$ Gaussian, independent of all $\alpha^\kappa_\ell \in \C^{S_\kappa}$. To estimate the $\psi_1$-norm of $\varphi_\ell$, we estimate with the Cauchy-Schwarz inequality and assumption $(c)$
	\begin{align*}
		\norm{\varphi_\ell}_2 \leq \abs{\gamma_\ell} \norm{b^i_\ell}_2 \left( \sum_{\kappa \in [r]} \norm{b^\kappa_\ell}_2\norm{\widehat{Y}_\kappa}_F\norm{ \alpha^\kappa_\ell}_2 \right) \leq \frac{\muplus \sqrt{ k_i}}{\sqrt{q}} \abs{\gamma_\ell}\sqrt{\sum_{\kappa \in [r]} \frac{\muplus^2 k_\kappa}{q} \norm{ \alpha^\kappa_\ell}_2^2} \norm{\widehat{\calY}}_F.
	\end{align*}
	\eqref{eq:OrliczIndep} implies that the  $\psi_1$-norm of this  expression is smaller than $$\tfrac{\muplus^2}{q} \sqrt{k^*} \norm{\widehat{\calY}}_F \normm{\sqrt{\sum_{\kappa \in [r]}\normm{ \alpha^\kappa_\ell}_2^2 k_\kappa}}_{\psi_2},$$
	where we used that univariate Gaussians have $\psi_2$-norm $\leqsim 1$. The term $\norm{\sqrt{\sum_{\kappa \in [r]}\normm{ \alpha^\kappa_\ell}_2^2 k_\kappa}}_{\psi_2}$ is in fact the $\psi_2$-norm of a vector $g$ with independent Gaussian entries, where for every $i=1, \dots, r$, $s_i$ of the entries have the variance $k_i$. \eqref{eq:GaussOrlicz} therefore implies that the expression is smaller than $\sqrt{ \sum_{i \in [r]} s_i k_i}$. We have secured the bound of $$R \leq\tfrac{\muplus^2}{q} \sqrt{k^*} \norm{\widehat{\calY}}_F \sqrt{ \sum_{i \in [r]} s_i k_i}$$ 
	for the application of Corollary \ref{cor:SwissNonRectangular}, and we move on to $\sigma^2$. \newline
	
	First, we have
	\begin{align*}
		\erw{\varphi_\ell^* \varphi_\ell }= \erw{\norm{b^i_\ell}_2^2\gamma_\ell^2 \sum_{\kappa, j \in [r]} (b^\kappa_\ell)^* \widehat{Y}_\kappa \alpha_\ell^\kappa (\alpha_\ell^j)^* \widehat{Y}_j^* b^j_\ell} = \norm{b^i_\ell}_2^2 \sum_{\kappa \in [r]} \sprod{b_\ell^\kappa (b_\ell^\kappa)^*, \widehat{Y}_\kappa \widehat{Y}_\kappa^*} \asymp \frac{\mu^2 k_i}{q}\sum_{\kappa \in [r]} \sprod{b_\ell^\kappa (b_\ell^\kappa)^*, \widehat{Y}_\kappa \widehat{Y}_\kappa^*}
	\end{align*}
	where we used that $\erw{\alpha^j_\ell (\alpha^\kappa_\ell)^*} = \delta_{j\kappa} \id_{S_\kappa}$. Taking the sum over $\ell \in [q]$, we obtain
	\begin{align*}
		\sum_{\ell \in [q]} \erw{\varphi_\ell^* \varphi_\ell } \leqsim \frac{\muplus^2 k^*}{q}\sum_{\kappa \in [r]}\sprod{\sum_{\ell \in [q]} b_\ell^\kappa (b_\ell^\kappa)^*, \widehat{Y}_\kappa \widehat{Y}_\kappa^*} = \frac{\muplus^2 k^*}{q} \norm{\widehat{\calY}}_F^2
	\end{align*}
	where we in the last step used assumption $(b)$. Similarly, one proves $\sum_{\ell \in [q]} \erw{\varphi_\ell^* \varphi_\ell } \geqsim \tfrac{\muminus^2 k_*}{q} \norm{\widehat{\calY}}_F^2$. $\sum_{\ell \in [q]} \erw{\varphi_\ell \varphi_\ell^*}$ is dealt with similarly: one  obtains
	\begin{align*}
		\norm{\sum_{\ell \in [q]} \erw{\varphi_\ell^* \varphi_\ell }}_{2\to 2} \leq  \sum_{\ell \in q} \norm{b^i_\ell(b^i_\ell)^*}_{2\to 2} \sum_{\kappa \in [r]}\sprod{b_\ell^\kappa (b_\ell^\kappa)^*, \widehat{Y}_\kappa \widehat{Y}_\kappa^*} &\leqsim  \frac{\muplus^2 k^*}{q} \sum_{\kappa \in [r]}\sprod{\sum_{\ell \in q} b_\ell^\kappa (b_\ell^\kappa)^*, \widehat{Y}_\kappa \widehat{Y}_\kappa^*} \\ &= \frac{\muplus^2 k^*}{q} \norm{\widehat{\calY}}_F^2.
	\end{align*}
	All in all, we have $ \tfrac{\muminus^2 k_*}{q} \norm{\widehat{\calY}}_F^2 \leqsim \sigma^2 \leqsim \tfrac{\muplus^2 k^*}{q} \norm{\widehat{\calY}}_F^2$.
	
	Towards applying Corollary \ref{cor:SwissNonRectangular}, let us note that the bounds we have proven together with assumption $(d)$ secure that
	\begin{align*}
		\sigma^2 + \tfrac{1}{2}R\log\left(1+ \tfrac{C_2R^2 q}{\sigma^2} \right) &\leqsim \tfrac{\mu_{+}^2k^*}{q} \norm{\widehat{\calY}}_F^2 + \tfrac{\mu_{+}^2}{q} \sqrt{k^*} \norm{\widehat{\calY}}_F \sqrt{\sum_{i\in [r]}s_ik_i} \log\left(1+ \tfrac{\muplus^4 k^*\norm{\widehat{\calY}}_F^2 \sum_{i\in [r]}s_ik_i}{\muminus^2 k_*\norm{\widehat{\calY}_F^2}}\right) \\
		&\leqsim \tfrac{\mu_{+}^2}{q} k^*s \log\left(1+ \tfrac{\muplus^4 k^* \sum_{i\in [r]}s_ik_i}{\muminus^2 k_*}\right) \leqsim \tfrac{1}{4}\log\left(\tfrac{n(k^*+1)}{\epsilon}\right)^{-1}.
	\end{align*}
	We used the inequality between geometric and arithmetic mean $ab \leq \tfrac{1}{2}\left(a^2+b^2\right)$,  $\norm{\widehat{\calY}}_F^2 \asymp s$ and $\sum_{i \in [r] } s_ik_i \leq sk^*$. We also have by the same argument
		\begin{align*}
		R\log\left(1 + \tfrac{C_2R^2q}{\sigma^2}\right) &\leqsim \tfrac{\muplus^2}{q} \sqrt{k^*}\norm{\calY}_F\sqrt{\sum_{i\in [r]}s_ik_i}\log\left(1+ \tfrac{\muplus^4 k^* \sum_{i\in [r]}s_ik_i}{\muminus^2 k_*}\right)  \\
		&\leqsim \tfrac{\muplus^2}{q}sk^*\log\left(1+ \tfrac{\muplus^4 k^* \sum_{i\in [r]}s_ik_i}{\muminus^2 k_*}\right) \leqsim \tfrac{1}{4}\log\left(\tfrac{n(k^*+1)}{\epsilon}\right)^{-1}.
		\end{align*} 	
	
	Applying Corollary \ref{cor:SwissNonRectangular}  together with a union bound  yields
	\begin{align*}
	 \prb{ \max_{(i,j)\notin \calS} \norm{Y_i(j)}_2 \geq \frac{1}{2}}	\leq 2n(k^* +1) P_{k^*+1,\sigma^2,R,q}(\tfrac{1}{2}) \leqsim \epsilon,
	\end{align*}
	and the proof is finished.
\end{proof}

{ \bf Funding:} This Research was funded by the Deutsche
Forschungsgemeinschaft (DFG) Grant KU 1446/18-1.

{\bf Conflicts of Interest:} The author declares that he has no conflict of interest.  

\bibliographystyle{abbrv}
\bibliography{/homes/numerik/flinth/Documents/bibliographyCSandFriendsMASTER}

\end{document}